\documentclass[11pt]{amsart}

\usepackage{amssymb,amsmath,amsthm}
\usepackage{graphics}
\usepackage{txfonts}
\usepackage{bbm}
\usepackage[margin,draft]{fixme}
\usepackage{enumerate}         % better lists

\theoremstyle{definition}
\newtheorem{definition}{Definition}[section]

\newtheorem{remark}[definition]{Remark}
\theoremstyle{plain}
\newtheorem{theorem}{Theorem}
\newtheorem{lemma}[definition]{Lemma}

\newcommand{\Rplus}{\mathbb{R}_{\ge 0}}

\newcommand{\PP}{\mathbb{P}}
\newcommand{\F}{\mathcal{F}}
\newcommand{\Econd}[2]{\mathbb{E}\left[\left.#1\right|#2\right]}
\newcommand{\Ex}[2]{\mathbb{E}^{#1}\left[#2\right]}                     % expectation with superscript
\newcommand{\E}[1]{\mathbb{E}\left[#1\right]}                     % expectation
\newcommand{\Excond}[3]{\mathbb{E}^{#1}\left[\left.#2\right|#3\right]}  % conditional expectation with superscript
\newcommand{\Ind}[1]{\mathbf{1}_{\left\{#1\right\}}}
\newcommand{\set}[1]{\left\{#1\right\}}

\title{Simple examples of pure-jump strict local martingales}
 \author{Martin Keller-Ressel}
 \address{TU Dresden, Department of Mathematics, 01062 Dresden, Germany}
 \email{martin.keller-ressel@tu-dresden.de}
 \thanks{I would like to thank Philip Protter and an anonymous referee for helpful comments and to acknowledge funding from the Excellence Initiative of the German Research Foundation (DFG) under Grant ZUK 64.}
\date{\today}
\begin{document}

\begin{abstract}
We present simple new examples of pure-jump strict local martingales. The examples are constructed as exponentials of self-exciting affine Markov processes. We characterize the strict local martingale property of these processes by an integral criterion and by non-uniqueness of an associated ordinary differential equation. Finally we show an alternative construction for our examples by an absolutely continuous measure change in the spirit of (Delbaen and Schachermayer, PTRF 1995).
\end{abstract}

\maketitle

\section{Introduction}

Strict local martingales, i.e. local martingales which are not true martingales have attracted the interest of researchers in probability theory and financial mathematics. Examples of strict local martingales can shed light on the demarcation line between true and local martingales which plays a role in fundamental results such as Girsanov's theorem. In financial mathematics strict local martingales illustrate the pathological behavior of markets which are free of arbitrage but where  market prices deviate from fundamental prices (cf. \cite{Loewenstein2000}) and put-call-parity ceases to hold (cf. \cite{Cox2005}). For these reasons strict local martingales are often regarded as mathematical models for asset price bubbles. 

Most known examples of strict local martingales are continuous processes. Exceptions are the examples \cite[Ex.~3.11]{Kallsen2008a}, \cite[Sec.~4]{mijatovic2011note} and the more recent construction of \cite{Protter2014} through filtration shrinkage. 
In this note we present several simple new examples of pure-jump strict local martingales. These examples are constructed as exponentials of self-exciting jump processes of affine type. In this context, self-excitement means that the intensity of jumps is proportional to the value of the process itself. It is widely accepted (cf. \cite{Sornette2014}) that self-excitement plays a crucial role in the appearance of real-world asset price bubbles, so that from the point of view of financial modeling this construction seems natural. It is also well-known that self-exciting behavior of certain affine processes can lead to explosion in finite time (cf. \cite{DFS2003}) and both the example of \cite{Kallsen2008a} and our examples are based on modifying an explosive affine process in just the right way -- in our case by exponential tilting -- such that explosion does no longer occur, but the martingale property fails despite being a local martingale.
In Section 2 we introduce a first example of such a strict local martingale and show in section 3 how this example can be generalized to a much larger class of processes. Finally, in section 4 we provide an alternative construction of the processes by an absolutely continuous measure change.  

\section{A pure-jump strict local martingale}\label{Sec:example}
Let a filtered probability space $(\Omega, \F, (\F_t)_{t \geq 0}, \PP)$ be given and assume that the filtration satisfies the usual conditions. Define the measure 
\begin{equation}\label{eq:mu}
\mu(d\xi) = \frac{1}{2\sqrt{\pi}} e^{-\xi}\xi^{-3/2}\,d\xi, \qquad \xi \in \Rplus
\end{equation}
and let the adapted c\`adl\`ag process $X$ be given as the Feller process with state space $\Rplus$ and with generator 
 \begin{equation}\label{eq:generator}
 \mathcal{A}f(x) = - \frac{x}{2} f'(x) + x \int_0^\infty \left(f(x + \xi) - f(x) - f'(x)\xi\right)\mu(d\xi)
 \end{equation}
for all $f$ from the core $C_c^\infty(\Rplus)$.  We assume that the starting point $X_0$ is deterministic and strictly positive. The existence of such a process $X$ is guaranteed by \cite[Thm.~2.7]{DFS2003} and the solution belongs to the class of conservative non-negative affine processes (see \cite[Lem.~9.2]{DFS2003}). From \cite[Thm.~2.12]{DFS2003} it is known that $X$ is a semimartingale\footnote{We refer to \cite{Jacod1987} for notation and terminology related to semimartingales.} with characteristics
\[B_t(\omega) = -\frac{1}{2} \int_0^t {X_{s-}(\omega)} ds, \quad C_t(\omega) = 0, \quad \nu(\omega,d\xi,ds) = X_{s-}(\omega) \mu(d\xi) ds\]
relative to the truncation function $h(x) = x$. Writing $J(\omega, d\xi,ds)$ for the random measure associated to the jumps of $X$  and 
 \[\overline J(\omega,X_{s-}, d\xi,ds) = J(\omega,d\xi,ds) - X_{s-} \mu(d\xi)ds\]
 for its compensated version, the canonical representation of $X$ (cf. \cite[Thm.~II.2.34]{Jacod1987}) is given by 
  \begin{equation}\label{eq:X}
 X_t = X_0  + \int_{\Rplus \times [0,t]} \xi \, \overline{J}(\omega,X_{s-}, d\xi,ds) - \frac{1}{2} \int_0^t X_{s-} ds.
 \end{equation}
This equation can be interpreted as an integral equation for $X$. Since the semimartingale characteristics determine the law of an affine process $X$ uniquely (cf. \cite[Thm.~2.12]{DFS2003}) we can consider $X$ as the weak solution, unique in law, of this integral equation.
Interpreting equation \eqref{eq:X} and taking into account the definition of $\overline J(\omega,X_{s-},d\xi,ds)$ we see that $X$ jumps upwards at a rate that is proportional to the current value $X_{t-}$, i.e. the jumps of $X$ are self-exciting and increase in rate as $X$ increases. This upward movement is counteracted by the compensator of the jump measure and an additional negative drift at rate $\frac{1}{2}X_{t-}$. The process $X$ has infinite jump-intensity but trajectories of a.s. finite variation since
 \begin{align*}
 \int_0^\infty \mu(d\xi) &=  \frac{1}{2\sqrt{\pi}} \int_0^\infty e^{-\xi}\xi^{-3/2} d\xi = \infty, \qquad \text{but}\\
 \int_0^\infty \xi \mu(d\xi) &=  \frac{1}{2\sqrt{\pi}} \int_0^\infty e^{-\xi}\xi^{-1/2} d\xi = \frac{1}{2}.
 \end{align*}
The latter equation implies by \cite[Prop.~II.2.29]{Jacod1987} that $X$ is in fact a \emph{special} semimartingale.
The following property makes $X$ interesting:
 \begin{theorem}\label{thm:strict}
 The process $S = e^X - 1$ is a strict local martingale.
 \end{theorem}
This theorem can be derived easily from the following Lemma.
\begin{lemma}\label{lem:mgf}The process $S = e^X - 1$ is a local martingale and 
\begin{equation}
\E{e^{u X_t}} = \exp\left\{X_0 \left(1 - \left(w(u)e^{-t/2} - 1\right)^2\right)\right\}, \qquad \text{for all} \qquad u < 1
\end{equation}
where $w(u) = 1 - \sqrt{1 - u}$.
\end{lemma}
Indeed, if the Lemma is true, then by monotone convergence
\[\E{S_t} = \E{e^{X_t}} - 1 = \lim_{u \to 1} \E{e^{u X_t}} - 1 = \exp\left(X_0 e^{-t/2}(2 - e^{-t/2})\right) - 1.\]
As a local martingale with non-constant expectation $S$ must then be a strict local martingale. 
\begin{proof}[Proof of Lemma~\ref{lem:mgf}] The main step is to compute the moment generating function $\E{e^{u X_t}}$; the local martingale property of $S$ will be obtained on the way. We could apply here directly the results of \cite{DFS2003} and \cite{KMayerhofer2013} but it is more instructive to do the computations explicitly. Let $T > 0$ and let $g: [0,T] \times (-\infty,1] \to (-\infty,1], (t,u) \mapsto g(t,u)$ be a function which for each $u \in (-\infty,1]$ is continuously differentiable in its first argument and which satisfies the initial condition $g(0,u) = u$. We will denote the time derivative of $g(t,u)$ by $\partial_t \, g(t,u)$. Applying Ito's formula for jump processes (cf. \cite{Protter2004}) to $M_t := \exp(g(T-t,u) X_t)$ yields
\begin{align*}
M_t &= M_0 + \text{loc. mg.}  - \int_0^t \partial_t \, g(T-s,u) X_{s-} M_{s-} ds - \frac{1}{2} \int_0^t g(T-s,u) X_{s-} M_{s-} ds + \\
&+ \sum_{s \le t} M_{s-} \left(e^{\Delta X_s g(T-s,u)} - 1 - \Delta X_s g(T-s,u)\right) = \\
&= M_0 + \text{loc. mg.} + \\
&+ \int_0^t \left(\int_0^\infty \left( e^{g(T-s,u)\xi} - 1 - g(T-s,u)\xi \right)\mu(d\xi) - \frac{1}{2} g(T-s,u) - \partial_t \, g(T-s,u) \right)X_{s-} M_{s-} ds,
\end{align*}
where \textit{loc. mg.} denotes a local martingale starting at zero which may change from line to line. Introducing the function
\begin{equation}
R(u) := \int_0^\infty \left(e^{u\xi} - 1 - u \xi\right)\mu(d\xi) - \frac{1}{2} u
\end{equation}
the above equation can be rewritten as
\begin{equation}\label{eq:M_simple}
M_t - M_0 = \text{loc. mg.} + \int_0^t \left(R(g(T-s,u)) - \partial_t \, g(T-s,u)\right) X_{s-} M_{s-} ds.
\end{equation}
In our example the function $R(u)$ is well-defined for all $u \in (-\infty,1]$ and can be computed explicitly by partial integration and reduction to a Gamma-type integral (cf. Lemma~\ref{lem:integral}), yielding
\[R(u) = (1-u)  - \sqrt{1-u}, \qquad u \in (-\infty,1].\]
We conclude from \eqref{eq:M_simple} that if $g$ is a function that is bounded above by $1$ and satisfies the ODE
\begin{equation}\label{Eq:Riccati}
\partial_t \, g(t,u) = R(g(t,u)), \qquad g(0,u) = u
\end{equation}
then $M_t =  \exp(g(T-t,u) X_t)$ is a local martingale for any $T > 0$. These conditions are in particular satisfied by the constant function $g_+(t,1) = 1$ and we conclude that 
$e^{X_t}$  -- and hence also $S_t$  -- is a local martingale. Since $e^X$ is non-negative it is a supermartingale and it follows that $\E{e^{X_\tau}} \le \E{e^{X_0}} < \infty$ for all finite stopping times $\tau \ge 0$.

Consider now the case when $g$ satisfies \eqref{Eq:Riccati} and is in addition bounded above by $(1 - \epsilon)$ for some $\epsilon > 0$. Setting $f(x) = x^{1/(1- \epsilon)}$ we see that 
\[\E{f(M_\tau)} = \E{f(\exp(g(T-\tau,u)X_\tau))} \le \E{f(\exp((1 - \epsilon)X_\tau))} = \E{e^{X_\tau}} \le \E{e^{X_0}} < \infty\]
for all stopping times $\tau \le T$. It follows from de la Vall\'ee-Poussin's theorem that the family 
\[\set{M_\tau: \tau\;\text{stopping time}, \quad \tau \le T}\] 
is uniformly integrable for any fixed $T > 0$ and hence that $M$ is a true martingale. 
It is easy to verify, that for $g(0,u) = u < 1$ the unique solution of the differential equation \eqref{Eq:Riccati} is given by
\[g(t,u) = 1 - \left(w(u)e^{-t/2} - 1\right)^2.\]
Clearly $g(t,u)$ is bounded uniformly away from $1$. We conclude that for $u < 1$ the process $M_t = e^{g(T-t,u)X_t}$ is a true martingale and hence
\[\E{e^{uX_T}} = \E{M_T} = M_0 = e^{g(T,u)X_0},\]
completing the proof.
\end{proof}
\begin{remark}
The reader familiar with affine processes recognizes in \eqref{Eq:Riccati} the \emph{generalized Riccati equation} (cf. \cite[Sec.~6]{DFS2003}) that is associated to $X$.
\end{remark}
\begin{remark}
Note that 
\[g_-(t,1) := \lim_{u \to 1-} g(t,u) = e^{-t/2}(2 - e^{-t/2})\]
is a solution of \eqref{Eq:Riccati} with initial condition $u = 1$ and the constant solution $g_+(t,1) \equiv 1$ is another one. This indicates that the strict local martingale property is related to the non-uniqueness of solutions to \eqref{Eq:Riccati} for the initial value $u = 1$, which will be confirmed in Theorem~\ref{thm:integrability}.
\end{remark}

%\begin{corollary}For a time-horizon $T$ the fundamental value $S^*$ of $S$ is given by
%\[S_t^* = \Econd{S_T}{\F_t} = \exp\left(X_t (2e^{-(T-t)/2} - e^{-(T-t)})\right) - 1.\]
%\end{corollary}

\section{Generalizations}
Let $\mu$ be a non-zero L\'evy measure on $\Rplus$ that satisfies $\int_1^\infty e^\xi \, \mu(d\xi) < \infty$. As above, we define a c\`adl\`ag Feller process $X$ with state space $\Rplus$ through its generator
 \begin{equation}\label{eq:generator}
 \mathcal{A}f(x) = - b x f'(x) + x \int_0^\infty \left(f(x + \xi) - f(x) - f'(x)\xi\right)\mu(d\xi),
 \end{equation}
where $b = \int_0^\infty (e^\xi - 1 - \xi) \mu(d\xi) > 0$ and $f$ from $C_c^\infty(\Rplus)$, the core of $\mathcal{A}$. We assume the starting point $X_0 > 0$ to be deterministic and remark that the finiteness of $b$ is guaranteed by the integrability condition on the measure $\mu(d\xi)$. As in Section~\ref{Sec:example}, existence and non-explosion of the process $X$ are guaranteed by the results of \cite{DFS2003}. We write $J(\omega,d\xi,ds)$ for the random measure associated to the jumps of $X$ and 
\begin{equation}\label{eq:J}
\overline J(\omega,X_{s-}, d\xi,ds) := J(\omega,d\xi,ds) - X_{s-} \mu(d\xi)ds.
\end{equation}
for its compensated version. By the same arguments as in Section~\ref{Sec:example}, $X$ satisfies a stochastic integral equation
 \begin{equation}\label{eq:X_general}
 X_t = X_0  + \int_{\Rplus \times [0,t]} \xi \, \overline{J}(\omega,X_{s-}, d\xi,ds) - b \int_0^t X_{s-} ds.
 \end{equation}
where $b = \int_0^\infty (e^\xi - 1 - \xi) \mu(d\xi) > 0$. Finally, we define 
\begin{equation}\label{eq:R_def}
R(u) = \int_0^\infty \left(e^{u\xi}  - 1  - u\xi\right)\mu(d\xi) - bu.
\end{equation}
Note that $R(1) = R(0) = 0$ holds and $R$ is a strictly convex function, such that $R(u) < 0$ for all $u \in (0,1)$. Moreover, $R$ is continuously differentiable on $(-\infty,1)$ and $R'(0) = -b < 0$.\\

The next result is `essentially known' in the literature on affine processes and parts of it can be found in \cite[Thm~2.5]{K2008a}, \cite[Rem.~4.5.iv]{Mayerhofer2011} and \cite[Thm.~3.2]{KMayerhofer2013}. We give a self-contained and simple proof that does not need prerequisites from the theory of affine processes.

\begin{theorem}\label{thm:integrability}
Let $R$ be given by \eqref{eq:R_def}. Then $S = e^X - 1$ is a local martingale and the following are equivalent:
\begin{enumerate}[(a)]
\item $S = e^X - 1$ is a strict local martingale
\item The ordinary differential equation 
\begin{equation}\label{Eq:Riccati_repeat}
\partial_t \, g(t) = R(g(t)), \qquad g(0) = 1
\end{equation}
has more than one solution taking values in $(-\infty,1]$.
\item The function $1/R(u)$ is integrable in a left neighborhood of $u=1$, i.e. there exists $\epsilon > 0$ such that 
\begin{equation}\label{eq:osgood}
-\int_{1 - \epsilon}^1\frac{d \eta}{R(\eta)} < \infty.
\end{equation}
\end{enumerate}
\end{theorem}
\begin{remark}
The equivalence of (b) and (c) is closely related to Osgood's criterion for non-uniqueness of autonomous scalar ordinary differential equations, see \cite{Osgood1898}.
\end{remark}
\begin{remark}
It is interesting to compare this theorem with known results for diffusion processes. Integral conditions similar to \eqref{eq:osgood} are known to characterize the strict local martingale property of one-dimensional diffusions, see e.g. \cite{engelbert1990functionals}, \cite{blei2009exponential}, \cite[Thm.~1.6]{Delbaen2002} and \cite[Cor.~4.3]{Mijatovic2012}. The strict local martingale property of general diffusion processes has been linked to non-uniqueness of solutions of the associated Kolmogorov PDE in e.g. \cite[Sec.~9.5]{Lewis2000}, \cite{Heston2007} and \cite{Bayraktar2012}.
\end{remark}

\begin{proof}
The local martingale property follows directly by applying Ito's formula, as in the proof of Lemma~\ref{lem:mgf}. 

We proceed to show the analytic part of the theorem, that is, the equivalence of (b) and (c). To see that (b) implies (c) consider the case that \eqref{Eq:Riccati_repeat} has multiple solutions taking values in $(-\infty,1]$. Since the constant function $g_+(t) \equiv 1$ is clearly a solution, there must be another solution, $g_-$ and a time point $t_1 > 0$ such that $g_-(t_1) = 1 - \epsilon$, for some $\epsilon > 0$. Moreover, set $t_0 = \sup \set{t \ge 0: g_-(t) = 1}$ and note that by continuity of $g_-$ it holds that $g_-(t_0) = 1$ and $t_0 < t_1$. Taking into account the properties of $R$, it follows that $R(g_-(t)) \neq 0$ for $t \in  (t_0,t_1]$ and we may integrate \eqref{Eq:Riccati_repeat} to obtain
\[-\int_{g_-(t_1)}^{g_-(t)} \frac{d\eta}{R(\eta)} = - \int_{t_1}^t \frac{g_-(s)\,ds}{R(g_-(s))} = t_1 - t\]
for all $t \in (t_0,t_1]$. Letting $t$ tend to $t_0$ this becomes
\[-\int_{1 - \epsilon}^0 \frac{d\eta}{R(\eta)} = t_1 - t_0 < \infty\]
and we have shown (c).

To show the reverse implication, assume that \eqref{eq:osgood} holds, and note that due to convexity of $R$
\[-\int_0^1 \frac{d\eta}{R(\eta)} \ge -\frac{1}{ R'(0)}\int_0^1 \frac{d\eta}{\eta} = +\infty.\]
Hence, $T(x) := -\int_x^1{d\eta/R(\eta)}$ defines a continuously differentiable, strictly decreasing function that maps $(0,1]$ onto $[0,\infty)$. By the inverse function theorem, there exists a strictly decreasing and continuously differentiable inverse function $g$, mapping $[0,\infty)$ to $(0,1]$ with $g(0) = 1$ and which satisfies $T(g(t)) = t$, i.e.
\[-\int_{g(t)}^1 \frac{d\eta}{R(\eta)} = t, \qquad t \ge 0.\]
Differentiating on both sides shows that $g$ solves \eqref{Eq:Riccati_repeat}. But $g$ is different from the obvious constant solution $g_+(t) \equiv 1$ and it follows that  \eqref{Eq:Riccati_repeat} has multiple solutions taking values in $(-\infty,1]$, showing (b).

We now turn to the implication from (b) to (a): Denote by $g_+(t) \equiv 1$ the constant solution to \eqref{Eq:Riccati_repeat} and let $g_-(t)$ be another solution taking values in $(-\infty,1]$. Since $g_-$ is different from $g_+$ there exists a $T > 0$ such that $g_-(T) < 1$. We set
\[M_t^+ = e^{X_t}, \qquad \text{and} \qquad M_t^- = e^{g_-(T-t) X_t}\]
for $t \in [0,T]$. Both processes have the same terminal value $M_T^+ = M_T^-$, but different initial values $M_0^+ > M_0^-$. Applying Ito's formula just as in the proof of Lemma~\ref{lem:mgf} we find that both $M^+$and $M^-$ are local martingales and -- being positive -- also supermartingales. Assume for a contradiction that $S = e^X - 1$ is a true martingale. Then also $M^+$ is a true martingale, and 
\[M_0^+ = \E{M_T^+} = \E{M_T^-} \le M_0^-\]
in contradiction to $M_0^+ > M_0^-$. We conclude that $M^+$ and hence also $S$ is a strict local martingale and (a) follows.

Finally we show by contraposition that (a) implies (c). Assume that (c) does not holds true and consider the differential equation
\begin{equation}\label{Eq:Riccati_again}
\partial_t \, g(t,u) = R(g(t,u)), \qquad g(0,u) = u, \qquad u \in (0,1)
\end{equation}
Since $R(0) = R(1) = 0$ and $R$ is continuously differentiable and negative on $(0,1)$ this equation has a unique strictly decreasing solution $g(t,u)$ for initial values $u \in (0,1)$. Fixing an arbitrary $t > 0$ and integrating \eqref{Eq:Riccati_again} this solution satisfies 
\[-\int_{g(t,u)}^u \frac{d\eta}{R(\eta)} = t.\]
Assume that $\liminf_{u \to 1-} g(t,u) = 1- \epsilon$ for some $\epsilon > 0$. Taking this limit in the equation above yields
\[-\int_{1 - \epsilon}^1 \frac{d\eta}{R(\eta)} = t < \infty,\]
but we have started with the assumption that \eqref{eq:osgood} does \emph{not} holds true. We conclude that $\liminf_{u \to 1-} g(t,u) = 1$ and hence in fact $\lim_{u \to 1-} g(t,u) = 1$ for all $t \ge 0$. Now fix some $T > 0$ and set 
\[M_t^u = \exp\left(g(T-t,u)X_t\right), \qquad u \in (0,1).\]
Proceeding by Ito's formula as in the proof of Lemma~\ref{lem:mgr} we see that each $M^u$ is a local martingale. Moreover, $g(t,u) \le u$ and setting $f_u(x) = x^{1/u}$ we obtain
\[\E{f_u(M^u_\tau)} = \E{f_u(\exp(g(T-\tau,u) X_\tau))} \le \E{f_u(\exp(u X_\tau))} = \E{e^{X_\tau}} \le \E{e^{X_0}} < \infty\]
for any stopping time $\tau \le T$. For $u \in (0,1)$ the function $f_u$ is convex and satisfies $\lim_{x \to \infty} f_u(x)/x = \infty$. Hence, by de la Vall\'ee-Poussin's theorem the family \[\set{M^u_\tau: \tau\;\text{stopping time}, \quad \tau \le T}\] 
is uniformly integrable for any $T >0$ and thus $M^u$ is a true martingale. Using monotone convergence we obtain
\[\Econd{e^{X_t}}{\F_s} = \Econd{\lim_{u \to 1-} e^{g(T-t,u)X_t}}{\F_s} = \lim_{u \to 1-} \Econd{M_t^u}{\F_s} = \lim_{u \to 1-} M_s^u = e^{X_s}\]
for all $0 \le s \le t \le T$. This shows that $e^X$ and hence $S$ is a true martingale and completes the proof.
%For each $u \in (0,1)$ consider the function
%\[f_u: [0,u) \mapsto \Rplus, \quad y \mapsto -\int_y^u \frac{d\eta}{R(\eta)}.\] 
%This is a strictly decreasing $C^1$-function with upper boundary value $f_u(u) = 0$. Using convexity of $R$ we obtain for the lower boundary value 
%\[f_u(0) \ge \frac{1}{b} \int_0^u \frac{d\eta}{\eta} = \infty.\]
%Hence $y \mapsto f_u(y)$ has a unique strictly decreasing inverse function $t \mapsto g(t,u)$ on $[0,\infty)$, which is also $C^1$. As inverse function $g$ solves
%\[-\int_{g(t,u)}^u \frac{d\eta}{R(\eta)} = t, \qquad t \ge 0.\]
%Differentiating this equation and rearranging yields
%\[\partial_t \, g(t,u) = R(g(t,u)), \qquad g(0,u) = u.\]
%Repeating the exact argument from the proof of Lemma~\ref{lem:mgf} we obtain that 
%\[\E{e^{uX_t}} = e^{X_0 g(t,u)}, \qquad u < 1.\]
%Now suppose that $1/R(u)$ is integrable in a left neighborhood of $u=1$. Then also the function
%\[f_1(y) = -\int_y^1 \frac{d\eta}{R(\eta)}\]
%is a well defined strictly decreasing $C^1$-function and moreover
%\[f_1(y) = \lim_{u \to 1-} f_u(y).\]
%Hence there exist an inverse function that satisfies 
%\[g(y,1) = \lim_{u \to 1-} g(y,u).\]
%We conclude that 
%\[\E{S_t} + 1 = \E{e^{X_t}} = \lim_{u \to 1} \E{e^{u X_t}} = \lim_{u \to 1} e^{X_0 g(t,u)} = e^{X_0 g(t,1)},\]
%which is non-constant and hence that $S$ is a strict local martingale.
\end{proof}

\begin{theorem}\label{thm:slowly}Let the measure $\mu$ be given by 
\begin{equation}
\mu(d\xi) = c e^{-\xi} \xi^{-\alpha} \ell(\xi)\,d\xi, \qquad \xi \ge 0
\end{equation}
where $c >0$, $\alpha \in (1,2)$, and $\ell$ is slowly varying at infinity and bounded at zero. Define the associated Poisson random measure and the process $X$ as in \eqref{eq:J} and \eqref{eq:X_general}. Then the process $S = e^X - 1$ is a strict local martingale.
\end{theorem}
\begin{remark}
This theorem shows that the strict local martingale property of $S$ is related to the right tail of the jump measure and hence determined only by the large jumps of $S$.
\end{remark}
\begin{proof}
First note that $\mu$ is a L\'evy measure and satisfies the integrability condition $\int_1^\infty e^\xi \,\mu(d\xi) < \infty$, due to \cite[Prop~1.5.10]{Bingham1989}. This condition implies by \cite[Lem.~9.2]{DFS2003} that $X$ is a conservative process. Denote by $f \sim g$ asymptotic equivalence of functions; we will indicate whether the equivalence holds at $0$ or at $\infty$. We set 
\[F(\xi) =  - \int_\xi^\infty e^{x}\mu(dx).\] Using that $\alpha > 1$ it follows from \cite[Prop.~1.5.10]{Bingham1989} that 
\[F(\xi) \sim \tfrac{c}{\alpha - 1} \xi^{1 - \alpha} \ell(\xi)\qquad \text{as $\xi \to \infty$}.\] By partial integration we can rewrite $R$ as
\begin{equation*}
R(1 - z) = z \int_0^\infty e^{-z\xi} F(\xi) d\xi + z \int_0^\infty \left(e^\xi - 1 \right) \mu(d\xi).
\end{equation*}
By Karamata's Tauberian theorem (cf. \cite[Thm.1.7.1]{Bingham1989}) for the Laplace transform it follows that 
\[R(1 - z) \sim \frac{c}{\Gamma(\alpha)} \frac{\pi}{\sin(\pi(\alpha - 1))} z^{\alpha - 1} \ell(1/z)\qquad \text{as $z \to 0+$.}\]
Setting $\tilde \ell(x) = 1/\ell(x)$, which is also a slowly varying function, we obtain
\[\frac{1}{x^2R(1 - 1/x)} \sim \frac{\Gamma(\alpha)}{c} \frac{\sin(\pi(\alpha - 1))}{\pi} x^{\alpha-3} \tilde \ell(x), \qquad \text{as $x \to \infty$.}\]
By \cite[Prop.1.5.10]{Bingham1989} and using that $\alpha < 2$ it follows that the integral
\[-\int_{1 - \epsilon}^1 \frac{du}{R(u)} = \int_{1/ \epsilon}^\infty \frac{dx}{x^2 R(1 - 1/x)}\]
converges, which implies by Theorem~\ref{thm:integrability} that $S = e^X - 1$ is a strict local martingale.
\end{proof}
Particularly tractable examples can be produced from Theorem~\ref{thm:slowly} by choosing 
\[c = \frac{\sin(\pi (\alpha-1))}{\pi}\Gamma(\alpha) , \qquad \ell \equiv 1.\]
In this case we obtain from Lemma~\ref{lem:integral}
\[R(u) = (1 - u) - (1-u)^{\alpha - 1}.\]
The corresponding generalized Riccati equation \eqref{Eq:Riccati} can be solved explicitly for the initial value $u = 1$ and we obtain that $S = e^X - 1$ is a strict local martingale with expectation
\[\E{S_t} = \exp\left(X_0 g_-(t,1)\right) - 1, \quad \text{where} \quad g_-(t,1) = 1 - \left(1 - e^{(\alpha-2)t}\right)^{\tfrac{1}{2-\alpha}}.\]

\section{Alternative Construction \`a la Delbaen-Schachermayer}
We present an alternative construction of the strict local martingale $S$ from a true martingale $M$ by inversion and an absolutely continuous measure change. This technique is most familiar in the case of the inverse Bessel process of dimension~3 and has been used to show its strict local martingale property (see e.g.~\cite[Chapter~VI.3]{Revuz1999}). It has been generalized to continuous semimartingales by Delbaen and Schachermayer in \cite{Delbaen1995}. Our example shows that the technique can be used to produce discontinuous strict local martingales as well.

Define the measure
\begin{equation}
\tilde \mu(d\xi) = \frac{1}{2 \sqrt{\pi}} \xi^{-3/2}d\xi
\end{equation}
and note that $\tilde \mu$ is related to $\mu$ from \eqref{eq:mu} by $\mu(d\xi) = e^{-\xi} \tilde \mu(d\xi)$. We define $\tilde X$ as the Feller process with state space $[0,\infty]$ and with generator given by
\begin{equation}\label{eq:generator2}
\tilde{\mathcal{A}}f(x) = - x f'(x) + x \int_0^\infty \left(f(x + \xi) - f(x)\right)\tilde \mu(d\xi)
\end{equation}
for $f$ from the core $C_c^\infty(\Rplus)$. We assume that $\tilde X$ starts at a deterministic strictly positive initial point $\tilde X_0 > 0$. The process $\tilde X$ is a small modification of the non-conservative affine process $Y$ in \cite[Example~9.3]{DFS2003} by adding a negative drift. In fact it can be directly related to $Y$ by the transformation $\tilde X_t = \int_0^t e^{-(t-s)}dY_s$. It follows that also $\tilde X$ is non-conservative, but instead explodes to $+\infty$ at a predictable stopping time $\tau$ which is finite with positive $\mathbb{P}$-probability. Note that in contrast to the processes considered in the previous sections, $\tilde X$ is \emph{not} a semimartingale.

Obviously, the time of explosion $\tau$ is announced by the sequence of first hitting times
\begin{equation}\label{eq:taun}
\tau_n := \inf \{t \ge 0: \tilde X_t \ge n\}.
\end{equation}
Using Lemma~\ref{lem:integral} it is easy to check that $h(x) = e^{-x}$ is a harmonic function for $\tilde X$, i.e. $\tilde{\mathcal{A}}(e^{-x}) = 0$ and we conclude that 
\[M_t = \exp\left(- \Big(\tilde X_t - \tilde X_0\Big)\right)\]
is a bounded local martingale and hence a true martingale. Note that $M_t$ reaches the value zero exactly at $\tau$ and remains zero afterwards. We define a measure $\mathbb{Q}$, absolutely continuous with respect to $\mathbb{P}$, by setting
\begin{equation}\label{eq:measure_change}
\left.\frac{d\mathbb{Q}}{d\mathbb{P}}\right|_{\F_t} = M_t.
\end{equation}
As $M$ reaches zero with positive $\mathbb{P}$-probability the measures $\mathbb{Q}$ and $\mathbb{P}$ are not equivalent. In particular $\mathbb{Q}(\tau \le t) = 0$ for every $t \ge 0$ and hence
\[\mathbb{Q}(\tau = \infty) = \lim_{t \to \infty} \mathbb{Q}(\tau > t) = 1,\]
while $\mathbb{P}(\tau = \infty) < 1$.

\begin{theorem}
The process $L_t = \frac{1}{M_t}\Ind{t < \tau}$ is a strict local martingale under $\mathbb{Q}$.
\end{theorem}
\begin{remark}
Since $\mathbb{Q}(t < \tau) = 1$, it is correct to write $L_t = \frac{1}{M_t}$ $\mathbb{Q}$-a.s. However, this equality does \emph{not} hold true $\mathbb{P}$-a.s.
\end{remark}

\begin{proof}
We first show that $L$ is a local martingale by using $(\tau_n)_{n \in \mathbb{N}}$ from \eqref{eq:taun} as the localizing sequence. Note that $\tau_n \to \tau$ $\mathbb{P}$-a.s. and hence also $\mathbb{Q}$-a.s. Moreover $\mathbb{Q}(\tau = \infty) = 1$ such that $(\tau_n)$ is indeed a localizing sequence under $\mathbb{Q}$. Now
\begin{align*}
\Excond{\mathbb{Q}}{L_{t \wedge \tau_n}}{\F_s} &= \frac{1}{M_{s \wedge \tau_n}}\Excond{\mathbb{P}}{\frac{M_{t \wedge \tau_n}}{M_{t \wedge \tau_n}}\Ind{(t \wedge \tau_n) < \tau}}{\F_s} = \\ &=  \frac{1}{M_{s \wedge \tau_n}} = 
%\frac{1}{M_{s \wedge \tau_n}}\Ind{(s \wedge \tau_n) < \tau} = 
L_{s \wedge \tau_n}
\end{align*}
holds $\mathbb{Q}$-a.s. which shows that $L$ is a $\mathbb{Q}$-local martingale.
Finally, for sufficiently large $t$ we have
\[\Ex{\mathbb{Q}}{L_t} = \Ex{\mathbb{P}}{\frac{M_t}{M_t}\Ind{t < \tau}} = \mathbb{P}(\tau < t) < 1\]
and it follows that $L$ is not a true $\mathbb{Q}$-martingale.
\end{proof}
\begin{remark}
As mentioned, this construction parallels the well-known construction of the inverse Bessel process of dimension $3$ (cf. \cite[Chapter~VI.3]{Revuz1999}) that has been generalized by \cite{Delbaen1995}. In the construction of the $3$-dimensional inverse Bessel process the process $M_{t \wedge \tau}$ is a one-dimensional $\mathbb{P}$-Brownian motion starting at one and stopped upon hitting zero; under $\mathbb{Q}$ the process $M$ becomes a $3$-dimensional Bessel process and $1/M$, the inverse $3$-dimensional Bessel process, is a strict local martingale under $\mathbb{Q}$.
\end{remark}

To see that the above construction is equivalent to the construction from Section~\ref{Sec:example} note that the measure change \eqref{eq:measure_change} is nothing else than the $h$-transform of $\tilde X$ with respect to the harmonic function $h(x) = e^{-x}$ (see e.g. \cite{Palmowski2000}) and we can compute the generator of $\tilde X$ under $\mathbb{Q}$ from the formula
\begin{align*}
\tilde{\mathcal{A}}^{\mathbb{Q}} &= h^{-1} \tilde{\mathcal{A}} (fh) = -x (f'(x) - f(x)) + x \int_0^\infty \left(f(x+\xi)e^{-\xi} - f(x)\right) \tilde \mu(d\xi) = \\
&= x \int_0^\infty \left(f(x+\xi) - f(x) - f'(x)\xi\right)\mu(d\xi) + \\ &+ x f(x) \left\{1 - \int_0^\infty \left(1 - e^{\xi}\right) \mu(d\xi)\right\} + xf'(x)\left\{\int_0^\infty \xi \mu(d\xi)  - 1\right\} = \\ 
&= - \frac{x}{2}f'(x) + x \int_0^\infty \left(f(x+\xi) - f(x) - f'(x)\xi\right)\mu(d\xi)
\end{align*}
where we have used Lemma~\ref{lem:integral} to evaluate the integrals in the second-to-last line. This is exactly the generator of $X$ from \eqref{eq:X}, i.e. $\tilde X$ under $\mathbb{Q}$ is identical to the original process $X$ in the sense of Feller processes. Note that this connection via exponential measure change between non-conservativeness and strict local martingale property has also been exploited in \cite{Kallsen2008a} and \cite{Mayerhofer2011}.

\appendix

\section{An integration formula}

\begin{lemma}\label{lem:integral}
Let $\alpha \in (1,2)$ and set $C(\alpha) = \sin((\alpha-1)\pi) \Gamma(\alpha) / \pi$.
Then
\begin{align}
C(\alpha)\int_0^\infty \left(e^{u\xi} - 1\right)e^{-\xi} \xi^{-\alpha} d\xi &=  1 - (1-u)^{\alpha-1} \label{eq:gamma1}\\
\intertext{for all $u \le 1$, and}
C(\alpha)\int_0^\infty e^{-\xi} \xi^{1-\alpha} d\xi &= \alpha - 1. \label{eq:gamma2}
\end{align}
\end{lemma}
\begin{remark}In Section~\ref{Sec:example} we use the particular case $C(3/2) = \tfrac{1}{2\sqrt{\pi}}$.
\end{remark}
\begin{proof}
Set $w= u-1 \le 0$. By partial integration and substitution
\begin{align*}
&\int_0^\infty \left(e^{u\xi} - 1\right)e^{-\xi}\xi^{-\alpha}d\xi = \frac{1}{\alpha - 1} \int_0^\infty \left(w e^{w\xi} + e^{-\xi}\right)\xi^{1-\alpha}d\xi = \\
&= \frac{1}{\alpha-1} \int_0^\infty e^{-r} r^{1-\alpha} dr \cdot \left(1 - (-w)^{\alpha - 1} \right) = \\
&= \frac{\Gamma(2 - \alpha)}{\alpha-1} \left(1 - (1-u)^{\alpha-1}\right).
\end{align*}
By Euler's reflection formula for the Gamma function
\[\frac{\Gamma(2-\alpha)}{\alpha - 1} = \frac{\pi}{\sin((\alpha -1)\pi) \Gamma(\alpha)} = \frac{1}{C(\alpha)}\]
and \eqref{eq:gamma1} follows. Moreover,
\[C(\alpha) \int_0^\infty e^{-\xi} \xi^{1-\alpha} d\xi = C(\alpha) \Gamma(2-\alpha)  = \alpha - 1\]
and also \eqref{eq:gamma2} follows.
\end{proof}

\bibliographystyle{alpha}%{ams-pln}	
\bibliography{references}
\end{document}